\documentclass[12pt,oneside]{article}
\usepackage{amsmath,amssymb,amsfonts,amsthm}
\usepackage{color}
\usepackage{graphicx}
\usepackage{color}
\usepackage{makeidx}

\newcommand{\Real}{\mathbb R}

\newcommand{\T}{{\mathcal T}}
\newcommand{\C}{{\mathcal C}}
\newcommand{\tm}{\T M}
\newcommand{\To}{\longrightarrow}

\def\pa{\partial}

\theoremstyle{plain}
\newtheorem{theorem}{Theorem}

\newtheorem{lemma}{Lemma}
\newtheorem{proposition}{Proposition}

\theoremstyle{definition}
\newtheorem{definition}{Definition}

\theoremstyle{remark}
\newtheorem{remark}{Remark}
\newtheorem{example}{Example}

\numberwithin{equation}{section}

\allowdisplaybreaks[4]

\begin{document}

\title{\bf{ {On Riemann curvature of spherically symmetric metrics}}}

\author{S. G. Elgendi }
\date{}

\maketitle
\vspace{-1.20cm}

\begin{center}
{Department of Mathematics, Faculty of Science,\\
Islamic University of Madinah, Madinah, Saudi Arabia}
\vspace{-8pt}
\end{center}


\begin{center}
salah.ali@fsc.bu.edu.eg, salahelgendi@yahoo.com
\end{center}

\vspace{0.3cm}

\begin{abstract}
  In this paper, studying the inverse problem, we establish a curvature compatibility condition on  a spherically symmetric Finsler metric. As an application, we characterize the spherically symmetric metrics of scalar curvature. We construct  a Berwald frame for a spherically symmetric Finsler surface and  calculate some associated geometric objects.   Several examples are provided and discussed.  Finally, we   give a note on a certain general  $(\alpha,\beta)$-metric which appears in the literature.
 \end{abstract}



\section{Introduction}

 In $\mathbb{R}^n$, a Finsler function $F$ on a convex domain $\Omega $ is considered spherically symmetric if $(\Omega,F)$ remains invariant under all rotations. Stated differently, $F$ is a spherically symmetric metric if $F(\mathcal{A} x,\mathcal{A} y)=F(x,y)$,  where $\mathcal{A} \in O(n)$, i.e., $F$ is orthogonally invariant (see \cite{Guo-Mo}). The spherically symmetric metrics were introduced by Rutz in \cite{Rutz}. In general relativity, Rutz extended the well-known Birkhoff theorem to the Finslerian area. The gravitational field's vacuum Einstein field equations can be solved using this class of Finsler metrics, for instance.  Many articles on spherically symmetric Finsler metrics have been published recently; as an example, we can consult { \cite{Gabran,Elgendi-SSM,Zhou_Mo,Sadeghia}.}

Let $|\cdot|$ be the standard Euclidean norm, and   $\langle \cdot, \cdot \rangle$ be the standard inner product on $\mathbb{R}^n$. A Finsler function $F$ on $\mathbb{B}^n(r_0)\subset\mathbb{R}^n$ is a spherically symmetric Finsler metric if, and only if there exists  a positive and smooth function $\phi $ such that $F(x,y)= u \phi(r,s)$, where   $r=|x|$, $u=|y|$,  and $s=\frac{\langle x, y \rangle}{|y|}$ for all $(x,y)\in T\mathbb{B}^n(r_0)\backslash \{0\}$.

 In this note, we construct a curvature compatibility condition on a spherically symmetric Finsler metric by studying the inverse issue of calculus of variation.   Given a spherically symmetric Finsler metric, the curvature compatibility criterion is given by
 \begin{equation*}
\phi_s R_1+    (s\phi+(r^2-s^2)\phi_s)R_{3} +\phi R_5 =0,
\end{equation*}
 where  $R_{1}, \ R_{3}$ and $R_{5}$ are given { below the relation   \eqref{Eq:R^i_j}}. Then, we classify    the spherically symmetric Finsler metrics  of scalar flag curvature. As an application of the curvature compatibility condition, we show that a spherically symmetric Finsler metric is  of scalar curvature  {$\textbf{K}=\textbf{K(x,y)}$} if and only if it is either   of dimension two or $R_1=K\phi^2$ and $R_3=0$.

For Finsler surfaces,  in   \cite{Berwald} Berwald has introduced a frame (called Berwald frame). The geometry of a Finslerian surface is completely determined in terms of  this frame. 
 We establish a Berwald frame for a spherically symmetric Finsler surface, moreover, we calculate its main scalar.

Now, let's turn our attention to a more general class than the spherically symmetric metrics called the general $(\alpha,\beta)$-metrics. For this purpose, let    $\alpha=\sqrt{a_{ij}y^iy^j}$ be a Riemannian metric and $\beta=b_i(x)y^i$ be a $1$-form on $M$.
 A   general $(\alpha,\beta)$-metric  $F$ is defined by  
 $$F=\alpha \varphi(b^2,s),\  s:=\frac{\beta}{\alpha}$$
 such that $\varphi$  is a positive smooth function, $ \|\beta_x\|_\alpha<b_0$ with some regularity conditions, for more details cf. \cite{Yu-Zhu}.

The class of general  $(\alpha,\beta)$-metric  generalizes interesting classes of Finsler metrics, for example, it generalizes the $(\alpha,\beta)$-metrics, namely, when $\varphi$ is independent of $b^2$. Also, it generalizes the spherically symmetric Finsler metrics when $\alpha=|y|$ is the standard  Euclidean norm, $\beta=\langle x , y \rangle$ is the standard  Euclidean inner product and $b^2=|x|^2$.

 In this part, we give a note on   { a general}  $(\alpha,\beta)$-metric  $F=\alpha \varphi(b^2,s)$, where $\alpha$ is a Riemannian metric,  $\beta=b_i(x)y^i$ is a $1$-form on the manifold $M$, $b^2=b_ib^i$, $b^i=a^{ij}b_j$,  $a^{ij}$ is the inverse of the Riemannian metric $\alpha$ and $s=\frac{\beta}{\alpha}$. For any dimension, we have the following equivalences:
 \begin{description}
 	\item[(a)] $ \varphi-s\varphi_s=0\Longleftrightarrow\varphi=f(b^2)s,$
 	\item[(b)]$\varphi\varphi_s-s(\varphi_s^2+\varphi \varphi_{ss})=0 \Longleftrightarrow \varphi(b^2,s)=f_1(b^2)s+f_2(b^2)\sqrt{b^2-s^2},$
 \end{description}
  {where} $ f_1$   and $ f_2$  are  arbitrary functions of $b^2$. Moreover, the metrics obtained in these equivalences lead to a degenerate metric; that is, $\det(g_{ij})=0$. Therefore,    such choices must be excluded.  The   {case (b) is} mentioned in several articles, for example, we refer to \cite{Theorem-1.1-Shen,Taybi,Example-9.2}.

\section{Preliminaries}

Let $M$ be  {an} $n$-dimensional manifold, the slit tangent bundle of nonzero tangent vectors to be $(\T M,\pi,M)$, and the tangent bundle to be $(TM,\pi_M,M)$.  On the base manifold $M$, let $(x^i) $ represents a local coordinate system, and on $TM$, let $(x^i, y^i)$ represent the induced coordinates.  $J = \frac{\partial}{\partial y^i} \otimes dx^i$ defines the vector $1$-form called the natural almost-tangent structure $J$ of $T M$. Canonical or Liouville vector field is the vector field $\C=y^i\frac{\partial}{\partial y^i}$ on $TM$.

{Let  $X$ be a vector field on $M$,  then   $ i_{X}$  (resp. $\mathcal{L}_X$)  denote the interior
product  (resp. the Lie derivative)  with respect to $X$. Let $f$ be a smooth real-valued function  on $M$, then $df$ is the differential of $f$.   A vector
$k$-form $L$ can be associated with two graded derivations $i_L$ and $d_L$ of the Grassman algebra of $ M$  such that
$$i_Lf=0, \,\,\,\, i_Ldf=df\circ L,\,\,\,\,\,\,  f \ \text{is a smooth function on $M$},$$
$$d_L:=[i_L,d]=i_L\circ d-(-1)^{k-1}di_L.$$}

A vector field $S\in \mathfrak{X}(\T M)$ that has the properties $JS = \C$ and $[\C, S] = S$ is a spray on $M$. Locally, $S$ has the form 
\begin{equation*}
  \label{eq:spray}
  S = y^i \frac{\partial}{\partial x^i} - 2G^i\frac{\partial}{\partial y^i},
\end{equation*}
where $G^i=G^i(x,y)$ are positively $2$-homogeneous
functions in   $y$ and called the \emph{spray coefficients}.

  A nonlinear connection is an $n$-dimensional distribution $H : u \in \tm \rightarrow H_u\subset T_u(\tm)$  supplementary distribution to the vertical distribution. Therefore,    for any $u \in \tm$, we have
\begin{equation}
\label{eq:direct_sum}
T_u(\tm) = H_u(\tm) \oplus V_u(\tm).
\end{equation}
For a given spray $S$, we have a canonical nonlinear connection $\Gamma$ given by $\Gamma = [J,S]$. Hence, we have  the   horizontal and vertical projectors
\begin{equation*}
  \label{projectors}
    h=\frac{1}{2}  (Id + [J,S]), \,\,\,\,            v=\frac{1}{2}(Id - [J,S])
\end{equation*}
  Locally, the   projectors $h$ and $v$ are given as follows
$$h=\frac{\delta}{\delta x^j}\otimes dx^j, \quad\quad v=\frac{\partial}{\partial y^j}\otimes \delta y^j,$$
$$\frac{\delta}{\delta x^j}=\frac{\partial}{\partial x^j}-G^h_j(x,y)\frac{\partial}{\partial y^h},\quad \delta y^j=dy^j+G^j_h(x,y)dx^h, \quad G^h_j(x,y)=\frac{\partial G^h}{\partial y^j}.$$
Moreover, the coefficients $G^h_{ij}$ of   Berwald connection are given by
$ G^h_{ij}=\frac{\partial G^h_j}{\partial y^i}.$

{ 
 The components $R^i_j$ of the Riemann curvature $R$ are given as follows
\begin{equation}
  \label{eq:16}
  R= R^i_j \,  dx^j \otimes \frac{\partial}{\partial y^i},
  \qquad   R^i_j =   2\frac{\partial G^i}{\partial x^j} - S(G^i_j) -
  G^i_k G^k_j.
\end{equation}}
Throughout  we use the notations
$$\pa_j:=\frac{\partial}{\partial x^j}, \quad \dot{\partial}_j=\frac{\partial}{\partial y^j}.$$
\begin{definition}
An $n$-dimensional Finsler manifold is a pair $(M,F)$, where $M$ is a  differentiable manifold of dimension $n$ and $F$ is a function   $F: TM \To \Real   $  satisfying the following conditions:
 \begin{description}
    \item[(a)] $F$ is    strictly positive and smooth on $\T M$,
    \item[(b)] $F(x,\lambda y)=\lambda F(x,y)$ for all $\lambda \in \mathbb{R}^+$, that is,  $F$ is positively homogeneous of degree $1$ in $y$,
            \item[(c)] The matrix   $g_{jk}= \dot{\partial}_j \dot{\partial}_k E$ has rank $n$ on $\T M$, where $g_{jk}$  are the components of the metric tensor and $E:=\frac{1}{2}F^2$ is  called the energy function of $F$.
 \end{description}
 \end{definition}
By making use of the Euler-Lagrange equation
\begin{equation*}
  \label{eq:EL}
  i_Sdd_JE=-dE,
\end{equation*}
    there exists  a unique spray $S$ on $TM$  called the
{geodesic spray} of the corresponding Finsler metric $F$.
\begin{definition}
 Assume that  $S$ is a spray on a manifold $M$, then $S$ is said to be  Finsler metrizable  if there
  is   a Finsler metric $F$ whose  geodesic spray     $S$. By \cite{MZ_ELQ}, we can conclude that a given  spray  $S$ is   Finsler metrizable if and only if there is a Finsler metric $F$ such that the following system is satisfied
   \begin{equation}
  \label{metrizable_system}
  d_hF=0, \quad d_\C F=F.
  \end{equation}
  \end{definition}

\section{Spherically symmetric metrics}

Consider  the standard Euclidean norm  $|\cdot|$ and the standard inner product $\langle \cdot , \cdot \rangle$ on $\mathbb{R}^n$.  Then, a Finsler metric $F$ on $\mathbb{B}^n(r_0)\subset\mathbb{R}^n$   is  spherically symmetric Finsler metric if and only if there is a   function $\phi:[0,r_0)\times\mathbb{R}^n\to \mathbb{R}$ which is positive,    smooth and satisfies
$F(x,y)=|y|\phi\left(|x|,\frac{\langle x, y \rangle}{|y|}\right)$
where $(x,y)\in T\mathbb{B}^n(r_0)\backslash \{0\}$. Or simply we write $F=u\ \phi(r,s)$ where $r=|x|$, $u=|y|$ and $s=\frac{\langle x, y \rangle}{|y|}$.

It is clear that the class of spherically symmetric metrics is a special class of  general $(\alpha,\beta)$-metrics. Moreover,   the class of general $(\alpha,\beta)$-metrics is introduced   by C. Yu et al. in \cite{Yu-Zhu}. Therefore, one can conclude that a  spherically symmetric metric  $F=u\phi(r,s)$  is  regular if and only if $\phi$ is a positive   smooth function,  and
\begin{equation}
\label{Regular_condition}
\phi-s\phi_s>0,\quad \phi-s\phi_s+(r^2-s^2)\phi_{ss}>0
\end{equation}
if $n\geq 3$, and  $\phi-s\phi_s+(r^2-s^2)\phi_{ss}>0$  if $n=2$ for all $|s|\leq r<r_0$.
Throughout, the subscript $s$ (resp. $r$) refers to the derivative with respect to $s$ (resp. $r$).
For more details  we refer, for example,   to \cite{Guo-Mo,Zhou_Mo}.

{  Since the components of the metric tensor associated with the Euclidean norm are   the Kronecker delta $\delta_{ij}$, then we lower  the index $i$ in $y^i$ and $x^i$ as follows }
$$y_i:=\delta_{ih} y^h, \quad x_i:=\delta_{ih} x^h.$$
Notice that $y_i$ and $x_i$ are the same as $y^i$ and $x^i$ respectively. Moreover, we have
$$y^iy_i=u^2, \quad x^ix_i=r^2, \quad y^ix_i=x^iy_i=\langle x,y \rangle.$$
The components $g_{jk}$ of the metric tensor of a the spherically symmetric metric $F=u \phi(r,s)$ are written in the form 
\begin{align}
\label{Eq:g_ij}
g_{jk}=&\sigma_0\ \delta_{jk}+\sigma_1x_jx_k+\frac{\sigma_2}{u} (x_jy_k+x_ky_j)+\frac{\sigma_3}{u^2}y_jy_k,
\end{align}
where  $$\sigma_0=\phi(\phi-s\phi_s),\ \ \sigma_1= \phi_s^2+\phi\phi_{ss},\ \ \sigma_2= (\phi-s\phi_s)\phi_s-s\phi\phi_{ss},  \ \ \sigma_3= s^2\phi\phi_{ss}-s(\phi-s\phi_s)\phi_s.$$
We maintain consistency in the indices in this way.    One can find   articles in the literature considering  the spherically symmetric Finsler metrics  with some kind of inconsistency  with the indices.  For example, in \cite{Zhou_Mo}, one can see that in a tensorial equation there is an up index in one side and  the same index is a lower index in the other side.

The following are the components $g^{jk}$ of the inverse metric tensor 
\begin{align}
\label{Eq:g^ij}
g^{jk}=&\rho_0\delta^{jk}+ \frac{\rho_1}{u^2}y^jy^k+ \frac{\rho_2}{u} (x^jy^k+x^ky^j)+\rho_3x^jx^k,
\end{align}
where
$$
 \rho_0=\frac{1}{\phi(\phi-s\phi_s)}, \quad \rho_1=\frac{(s\phi+(r^2-s^2)\phi_s)(\phi\phi_s-s\phi_s^2-s\phi\phi_{ss})}{\phi^3(\phi-s\phi_s)(\phi-s\phi_s+(r^2-s^2)\phi_{ss})},$$
 $$\rho_2=-\frac{\phi\phi_s-s\phi_s^2-s\phi\phi_{ss}}{\phi^2(\phi-s\phi_s)(\phi-s\phi_s+(r^2-s^2)\phi_{ss})},\quad \rho_3=-\frac{\phi_{ss}}{\phi(\phi-s\phi_s)(\phi-s\phi_s+(r^2-s^2)\phi_{ss})}.$$
  The coefficients  $G^h$ of the geodesic spray  of $F$ are  given by
\begin{equation}\label{G}
  G^h=uPy^h+u^2 Qx^h,
\end{equation}
where the functions $P$ and $Q$ are given by
\begin{equation}\label{P,Q}
 Q:=\frac{1}{2 r}\frac{ -\phi_r+s\phi_{rs}+r\phi_{ss}}{\phi-s\phi_s+(r^2-s^2)\phi_{ss}}, \quad P:=-\frac{Q}{\phi}(s\phi +(r^2-s^2)\phi_s)+\frac{1}{2r\phi}(s\phi_r+r\phi_s).
\end{equation}

\section{Curvature compatibility condition}

  The following identities are useful for subsequent use.
\begin{equation}
\label{Eq:derivatives}
    \begin{split}
        \frac{\partial r}{\partial x^j}= \frac{1}{r}x_j, \quad  \frac{\partial u}{\partial y^j}= \frac{1}{u}y_j,  \quad  \frac{\partial u}{\partial x^j}= 0,         \quad  \frac{\partial r}{\partial y^j}= 0,  \quad
         \frac{\partial s}{\partial x^j}= \frac{1}{u}y_j, \quad  \frac{\partial s}{\partial y^j}= \frac{1}{u}(x_k-\frac{s}{u}y_k).   \\
    \end{split}
\end{equation}
The coefficients  $G^i_j$ of the nonlinear connection   of a spherically symmetric metric $F$ are derived as follows
\begin{equation}
\label{Eq:G^i_j}
G^i_j=u P \delta^i_j+P_s x_jy^i+  \frac{1}{u}(P-sP_s)  y_jy^i+u Q_s x^ix_j+(2Q-sQ_s) x^iy_j.
\end{equation}
The components of the Riemann curvature (Jacobi endomorphism) $R^i_j$  are given by \cite{Zhou_Mo} as follows
\begin{equation}
\label{Eq:R^i_j}
R_{j}^{i}=u^{2}R_{1} \delta^{i}_{j}+R_{2}  y^{i} y_{j}+u^{2}R_{3}  x^{i} x_{j}+uR_{4}  x^{i} y_{j}+uR_{5} x_{j} y^{i}
\end{equation}
where
\begin{align*}
R_{1}=& 2 Q-\frac{s}{r} P_{r}-P_{s}+2\left(r^{2}-s^{2}\right) P_{s} Q+P^{2}+2 s P Q \\
R_{2}=& P_{s}-\frac{s}{r} P_{r}+\frac{s^{2}}{r} P_{r s}+s P_{s s}-2 Q+s Q_{s}-2 s P P_{s}-4 s P Q+4 s^{2} P_{s} Q-P^{2} \\
&-2 s\left(r^{2}-s^{2}\right) P_{s s} Q+3 s P P_{s}+s^{2} P Q_{s}+\left(r^{2}-s^{2}\right) s P_{s} Q_{s}-2 r^{2} P_{s} Q \\
R_{3}=& \frac{2}{r} Q_{r}-Q_{s s}-\frac{s}{r} Q_{r s}+2\left(r^{2}-s^{2}\right) Q Q_{s s}+4 Q^{2}-\left(r^{2}-s^{2}\right) Q_{s}^{2}-2 s Q Q_{s} \\
R_{4}=&-\frac{2 s}{r} Q_{r}+\frac{s^{2}}{r} Q_{r s}+s Q_{s s}-2\left(r^{2}-s^{2}\right) s Q Q_{s s}+\left(r^{2}-s^{2}\right) s Q_{s}^{2}-4 s Q^{2} \\
&+2 s^{2} Q Q_{s} \\
R_{5}=& \frac{2}{r} P_{r}-\frac{s}{r} P_{r s}-P_{s s}-Q_{s}+2 P Q-2 s P_{s} Q+2\left(r^{2}-s^{2}\right) P_{s s} Q-P P_{s} \\
&-s P Q_{s}-\left(r^{2}-s^{2}\right) P_{s} Q_{s}.
\end{align*}
For a   spherically symmetric Finsler metric $F=u \phi$, we have the geodesic spray \eqref{G}  determined  by    the functions $P$ and $Q$.  But for given functions $P$ and $Q$, we have to solve the inverse    problem to find the   Finsler metric whose    geodesic spray is given  by $P$ and $Q$.   In \cite{Elgendi-SSM},   two metrizability  conditions are investigated as follows.
\begin{lemma}\cite{Elgendi-SSM}
Le  $P(r,s)$ and $Q(r,s)$ be given, then the Finsler metric $F=u\phi(r,s)$ whose geodesic spray is given by $P$ and $Q$ is determined by   $\phi$ that satisfies the     metrizability conditions:
\begin{equation}
\label{Comp_C_C_2}
    \begin{split}
       C_1:= &(1+sP-(r^2-s^2)(2Q-sQ_s))\phi_s+(s P_s-2P-s(2Q-sQ_s))\phi =0,   \\
        C_2:=& \frac{1}{r}\phi_r-(P+Q_s(r^2-s^2))\phi_s-(P_s+sQ_s) \phi =0.
    \end{split}
\end{equation}
\end{lemma}
In \cite{MZ_ELQ}, for a Finsler function $F$ with the associated $h$ and $R$,  Musznay showed that $d_hF=0$ implies   $d_RF=0$. Hence,  we introduce a   compatibility condition on a spherically symmetric Finsler metric as follows:
\begin{proposition}
\label{Prop: curv_comp}
For a  spherically symmetric Finsler metric $F=u\phi(r,s)$, the curvature compatibility condition  on $F$ is given by
\begin{equation}
\label{Comp_C_3}
        C_3: = \phi_s R_1+    (s\phi+(r^2-s^2)\phi_s)R_{3} +\phi R_5 =0.
\end{equation}
\end{proposition}

\begin{proof}
Since the geodesic spray of $F=u\phi$ has the Riemann curvature $R$, then $d_RF=0$. Locally, this condition reads
 \begin{equation}\label{Eq:d_RF=0}
 R^i_j \frac{\partial F}{\partial y^i}=0.
 \end{equation}
 We have
 $$\frac{\partial F}{\partial y^i}=\frac{\phi}{u} y_i+\phi_s \left(x_i-\frac{s}{u}y_i\right), \quad x^i\frac{\partial F}{\partial y^i}=s\phi+(r^2-s^2)\phi_s. $$
Plugging the above formulae together with \eqref{Eq:R^i_j} into the condition \eqref{Eq:d_RF=0}, we have the following
\begin{align}\label{Eq:4.6}
\nonumber  u^{2}R_{1} \left( \frac{\phi}{u} y_j+\phi_s\left(x_j-\frac{s}{u}y_j\right)\right)&+ u\phi R_{2}    y_{j}+u^{2}R_{3} (s\phi+(r^2-s^2)\phi_s)   x_{j}\\ &+uR_{4}  (s\phi+(r^2-s^2)\phi_s) y_{j}+u^2\phi R_{5} x_{j}=0
\end{align}
By collecting the coefficients of $y_j$ and $x_j$, we get the following
\begin{align*}
&u \left(  (\phi-s\phi_s)R_{1}   + \phi R_{2}  +   (s\phi+(r^2-s^2)\phi_s) R_{4} \right) y_{j}\\
&+u^2\left( \phi_s R_1+    (s\phi+(r^2-s^2)\phi_s)R_{3} +\phi R_5\right) x_{j}=0.
\end{align*}
{Now, assume that we have two functions $\chi$ and $\psi$  such that
$$\chi  x_i+\psi  y_i=0.$$ 
Then, by contraction   by $x^i$ and then by $y^i$, taking the fact that $u\neq0$ into account, we   obtain  the  equations 
$$r^2 \chi +su\psi=0, \quad s\chi+u\psi=0.$$
Multiplying the second equation by $s$ and by subtraction, we have  $(r^2-s^2) \chi=0$ which holds for all $r$ and $s$, that is, $\chi=0$,  and therefore $\psi=0$. Using this property, \eqref{Eq:4.6} implies  
   the two equations}
\begin{equation}
\label{Eq:}
    \begin{split}
        & (\phi-s\phi_s)R_{1}   + \phi R_{2}  +   (s\phi+(r^2-s^2)\phi_s)R_{4} =0,   \\
         & \phi_s R_1+    (s\phi+(r^2-s^2)\phi_s)R_{3} +\phi R_5 =0.
    \end{split}
\end{equation}
Using the facts that $R_4=-sR_3$ and $R_2=-R_1-sR_5$, {  then  we have
\begin{align*}
0&=(\phi-s\phi_s)R_{1}   + \phi R_{2}  +   (s\phi+(r^2-s^2)\phi_s)R_{4}\\
&=(\phi-s\phi_s)R_{1}   - \phi  (R_1+sR_5) -s   (s\phi+(r^2-s^2)\phi_s)R_{3}\\
&=-s(\phi_s R_1+    (s\phi+(r^2-s^2)\phi_s)R_{3} +\phi R_5).
\end{align*} 
Therefore, \eqref{Eq:} reduces to the condition}
\begin{equation*}
         \phi_s R_1+    (s\phi+(r^2-s^2)\phi_s)R_{3} +\phi R_5 =0.
\end{equation*}
This completes the proof.
\end{proof}

\subsection{Spherically symmetric metrics of scalar flag curvature}

  It is well known that a Finsler metric $F$ has scalar flag curvature $K(x,y)$ if and only if
$$
R_{j}^{i}=K F^{2}\left(\delta^{i}_{j}-\frac{y^{i}}{F}\frac{\partial F}{\partial y^j}\right).
$$
We define the following covector
\begin{equation}
\label{Eq:m_j}
n_j:= x_j-\frac{s}{u}y_j .
\end{equation}
Raising the index $j$ in $n_j$ by the Euclidean inverse metric $\delta^{ij}$, we get the vector ${\bf{n}}^i:=n_j \delta^{ij}=x^i-\frac{s}{u}y^i$.
\begin{lemma}\label{Prop:n_i}
The vector $\bf{n}^i$ and the covector $n_i$ satisfy the following properties:
\begin{description}
\item[(a)] $\bf{n}^i$ is orthogonal to $y^i$ with respect the Eculidean metric; that is, ${\bf{n}}^iy^j\delta_{ij}=0$ or equivalently, $y^i{n}_i=y_i{\bf{n}}^i=0$.

\item[(b)]  $x^in_i=x_i{\bf{n}}^i={\bf{n}}^in_i=r^2-s^2$.

\item[(c)] $\frac{\partial s}{\partial y^i}=\frac{1}{u}n_i$,  $\frac{\partial n_i}{\partial y^j}=\frac{s}{u^3}\left(y_iy_j-u^2\delta_{ij}\right)-\frac{1}{u^2}n_j  y_i$.

\item[(d)]$n_i\neq 0$ and $r^2-s^2\neq 0$.
\end{description}
\end{lemma}

\begin{proof}   {  The proof is  straightforward calculations, so we omit it.}
\end{proof}
{As an application of the curvature compatibility condition   \eqref{Comp_C_3}, we have the following theorem which obtained first by  \cite{Huang} but in a completely different treatment.} 
\begin{theorem}
\label{Th:scalar_curv}
Spherically symmetric Finsler metrics are of scalar curvature $K(x,y)$ if and only if  either they are of dimension two or $R_1=K\phi^2$ and $R_3=0$.
\end{theorem}
\begin{proof}
Assume that $F=u\phi$ is a spherically { symmetric} metric of scalar curvature $K$. Then we have
$$R^i_j=K F^{2}\left(\delta^{i}_{j}-\frac{y^{i}}{F}\frac{\partial F}{\partial y^j}\right)=u^2\phi^2 K\delta^i_j-\phi K (\phi y_j+u \phi_s n_j)y^i.$$
 Using the formula \eqref{Eq:R^i_j}, we can write the above condition for any spherically symmetric Finsler metric of scalar flag curvature as follows
 \begin{equation}
 \label{Eq:R=Scalar}
u^{2}(R_{1} -\phi^2K)\delta^{i}_{j}+R_{2}  y^{i} y_{j}+u^{2}R_{3}  x^{i} x_{j}+uR_{4}  x^{i} y_{j}+uR_{5} x_{j} y^{i}+\phi K (\phi y_j+u \phi_s n_j)y^i=0.
 \end{equation}
 Contracting the indices $i$ and $j$ in \eqref{Eq:R=Scalar} and using the facts that $R_4=-sR_3$, $R_2=-R_1-sR_5$, we get
\begin{equation}
\label{Eq:Scalar_1}
(n-1)R_1+(r^2-s^2)R_3=(n-1)\phi^2 K.
\end{equation}
 Contracting  \eqref{Eq:R=Scalar} by ${\bf{n}}^j$ implies
\begin{equation}
\label{Eq:Scalar_2}
u\left(- s ( R_1- \phi^2 K)+(r^2-s^2)R_5+\phi\phi_s K(r^2-s^2)\right)y^i+\left(u^2( R_1-\phi^2 K)+ u^2(r^2-s^2)R_3\right) x^i=0.
\end{equation}
From which we conclude the following two equations
\begin{equation}
\label{Eq:Scalar_3}
  R_1+  (r^2-s^2)R_3 =\phi^2 K.
\end{equation}
\begin{equation}
\label{Eq:Scalar_4}
   s   R_1-(r^2-s^2)R_5-(s\phi  +\phi_s (r^2-s^2)) \phi K=0.
\end{equation}
Subtracting \eqref{Eq:Scalar_1} and \eqref{Eq:Scalar_3} yields
\begin{equation}
\label{Eq:Scalar_5}
  (n-2)(R_1- \phi^2 K)=0.
\end{equation}
Hence, we have $n=2$ or $R_1=\phi^2 K$. If $R_1=\phi^2 K$, then by \eqref{Eq:Scalar_3} and Lemma \ref{Prop:n_i} (d), we get $R_3=0$.

Conversely, it is known that all two dimensional Finsler metrics are of scalar curvature. Now, let $F$ be a spherically symmetric metric with dimension $n>2$ and  $R_1=K\phi^2$, $R_3=0$. Then by substituting into \eqref{Eq:R^i_j}, we have
$$R_{j}^{i}=u^{2}R_{1} \delta^{i}_{j}+R_{2}  y^{i} y_{j}+uR_{5} x_{j} y^{i}
$$
Using \eqref{Comp_C_3} and  the fact that $R_1=-R_2-sR_5$, then $R_2=-K\phi^2+s K \phi\phi_s$ and $R_5=-K\phi\phi_s$. Therefore, by the fact that
 $\frac{\partial F}{\partial y^j}=\frac{\phi}{u} y_j+\phi_s \left(x_j-\frac{s}{u}y_j\right)$,
 we get
$$R_{j}^{i}=u^{2}K\phi^2 \delta^{i}_{j}+(-K\phi^2+s K \phi\phi_s)  y^{i} y_{j}-uK\phi\phi_s x_{j} y^{i}=K F^2\left(\delta^i_j-\frac{y^i}{F}\frac{\partial F}{\partial y^j}\right).$$
That is, $F$ is of scalar curvature.
\end{proof}
The following example  provides a spherically symmetric surface  with $R_1\neq K\phi^2$ and  $R_3\neq0$.
\begin{example}
 Let $F=u\phi(r,s)$, where
\begin{align*}
  \phi (r,s)
    &=\frac{1}{r^5}  \sqrt{r^{2}-s^{2}}\exp \left(  \frac{2 s  }{ \sqrt{r^{2}-s^{2}} } \right).
\end{align*}
Using, Maple program,  \eqref{P,Q} and the above formula of  $\phi(r,s)$,  we obtain that
\begin{equation}
\label{P_Q_from_phi}
P= -\frac{s}{r^2}-\frac{3}{4 r^2}\,\sqrt {{r}^{2}-{s}^{2}}, \quad Q=\frac{7}{8 r^2}-\frac{3 s^2}{r^4}-\frac{3s}{4 r^4}\,\sqrt {{r}^{2}-{s}^{2}}.
\end{equation}
Also, we have
$$R_1={\frac {25\,{r}^{2}-25\,{s}^{2}}{16\,{r}^{4}}},\quad R_2=-{\frac {25}{16\,{r}^{2}}},\quad R_3=-{\frac {25}{16\,{r}^{4}}},\quad R_4={\frac {25\,s}{16\,{r}^{4}}},\quad R_5={\frac {25\,s}{16\,{r}^{4}}}.$$
For this surface,  we have
$K=0$.
\end{example}

\section{A Berwald frame of a spherically symmetric surface }

Let $(\ell^i,m^i)$ be the Berwald frame  of a Finsler surface.
By \cite{Matsumoto_2D}, we have  the following lemma.
\begin{lemma}\cite{Matsumoto_2D}
The metric tensor $g_{ij}$ and the inverse metric $g^{ij}$ are given, in terms of the Berwald frame, as follows:
$$g_{ij}=\ell_i\ell_j+\varepsilon m_im_j,\quad g^{ij}=\ell^i\ell^j+\varepsilon m^im^j.$$
Moreover, the Cartan tensor $C_{ijk}$ is given by
$$FC_{ijk}=Im_im_jm_k,$$
where $I(x,y)$ is the main scalar of the surface $F=u\phi(r,s)$.
\end{lemma}
Let  $n^i:=g^{ij}n_j$, then we have the following proposition.
\begin{proposition}
The vector $n^i$ is given by
\begin{equation}\label{Eq:m^i}
n^i =\rho_0 {\bf{n}}^i+\frac{(r^2-s^2)}{u}(\rho_2 y^i+u\rho_3x^i).
\end{equation}
 Moreover, it
 satisfies the properties
$$n^i\ell_i=0,\quad n^in_i=\frac{r^2-s^2}{\phi(\phi-s\phi_s+(r^2-s^2)\phi_{ss})}.$$
\end{proposition}
\begin{proof}
Using the inverse metric \eqref{Eq:g^ij} together with the fact that $n^i=g^{ij}n_j$, we get \eqref{Eq:m^i}.
Now, since   $\ell_i=\frac{\partial (u\phi)}{\partial y^i}=\frac{\phi}{u} y_i+\phi_s n_i$, then we have
\begin{align*}
n^i\ell_i&=\left(\rho_0 {\bf{n}}^i+\frac{(r^2-s^2)}{u}(\rho_2 y^i+u\rho_3x^i)\right)\left(\frac{\phi}{u} y_i+\phi_s n_i\right)\\
&=(r^2-s^2)(\phi_s\rho_0 +\phi\rho_2+(s\phi +(r^2-s^2)\phi_s)\rho_3).
\end{align*}
By using the definition of $\rho_0$, $\rho_1$, $\rho_2$ and $\rho_3$, one can show that
$\phi_s\rho_0 +\phi\rho_2+(s\phi +(r^2-s^2)\phi_s)\rho_3=0.$
That is, $n^i\ell_i=0$.
We can calculate the scalar quantity $n^in_i$ as follows
$$n^in_i=(r^2-s^2)(\rho_0 +(r^2-s^2)\rho_3).$$
One can show that  
$\rho_0 +(r^2-s^2)\rho_3=\frac{1}{\phi(\phi-s\phi_s+(r^2-s^2)\phi_{ss})}.$
That is
$n^in_i=\frac{r^2-s^2}{\phi(\phi-s\phi_s+(r^2-s^2)\phi_{ss})}.$
\end{proof}
 Let $m_i=a(r,s)n_i$,  where
 \begin{equation}
 \label{Eq:a}
 a(r,s)=\sqrt{\frac{\phi(\phi-s\phi_s+(r^2-s^2)\phi_{ss})}{r^2-s^2}}.
 \end{equation}
   Now, we have the following theorem.
\begin{theorem}
The frame $(\ell^i,m^i)$ constitutes an orthonormal frame (the Berwald frame) for spherically symmetric surfaces, where $m^i:=g^{ij}m_j=an^i$, that is;
\begin{equation}
\label{Eq:Berwald_frame}
\ell^i=\frac{y^i}{F}, \quad  m^i =a\left( \rho_0 {\bf{n}}^i+\frac{(r^2-s^2)}{u}(\rho_2 y^i+u\rho_3x^i)\right).
\end{equation}
\end{theorem}
\begin{proof}
Keeping in mind that $n_i$ and $\ell_i$ are orthogonal.   To build the Berwald frame for a spherically symmetric Finsler surface, we have to choose a vector $m_i$ in the direction of $n_i$ as follows $m_i=a(r,s) n_i$ so that the length of $m_i$ is the unit. That is,
$$m^2:=m^im_i=a^2n^in_i=\frac{a^2(r^2-s^2)}{\phi(\phi-s\phi_s+(r^2-s^2)\phi_{ss})}=1.$$
Therefore, we have
 $$m_i=\sqrt{\frac{\phi(\phi-s\phi_s+(r^2-s^2)\phi_{ss})}{r^2-s^2}} n_i.$$
Since $y^im_i=0$ then $\ell^im_i=\frac{y^i}{F}m_i=\frac{y^i}{u \phi}m_i=0$. Hence, $\ell^i$ is orthogonal to $m_i$ and $m^2=1$, then   the Berwald frame is given by $(\ell^i,  m^i)$.
\end{proof}
\begin{remark}
For the two-dimensional case,  Berwald \cite{Berwald} introduced his  frame for the positive definite metrics.  Later  Basco and Matsumoto    \cite{Matsumoto_2D} modified Berwald frame to cover also the non positive definite metrics. The  modified frame   is given by $(\ell^i,m^i)$, where  $m^i$ is a vector which is orthogonal to the supporting element $\ell_i$ and the co-frame is $(\ell_i,m_i)$. Moreover, we have
 $$m_i=g_{ij}m^j, \quad m^im_i=\varepsilon, \quad \ell^im_i=0,$$
 where $\varepsilon=\pm 1$ and the sign $\varepsilon$ is  called the signature of $F$. In the positive definite case, $\varepsilon=+1$ and otherwise, $\varepsilon=-1$.
\end{remark}

\subsection{The Main scalar of a spherically symmetric surface}

The Cartan tensor of of a spherically symmetric surface can be calculated as follows:
\begin{proposition}\label{Prop:Cartan_tensor}
The components  $C_{ijk}$ of Cartan tensor of a spherically symmetric Finsler metric are  given by
\begin{align*}
C_{ijk}=&\frac{\mu}{2u}(x_i\delta_{jk}+x_j\delta_{ik}+x_k\delta_{ij})+\frac{\nu}{2u}x_ix_jx_k-\frac{s \mu}{2u^2}(y_i\delta_{jk}+y_j\delta_{ik}+y_k\delta_{ij})+\frac{3s\mu-s^3\nu}{2u^4}y_iy_jy_k\\
&+\frac{s^2\nu-\mu}{2u^3}(y_iy_jx_k+y_jy_kx_i+y_iy_kx_j)-\frac{s\nu}{2u^2}(x_ix_jy_k+x_ix_ky_j+x_kx_jy_i),
\end{align*}
where, for simplicity, we use the notations $\mu:=\phi \phi_s-s\phi_s^2-s\phi\phi_{ss}$ and $\nu:=3\phi_s\phi_{ss}+\phi\phi_{sss}$.
\end{proposition}

\begin{proof}
According to Lemma \ref{Prop:n_i} (c),  $\frac{\partial s}{\partial y^k}=\frac{1}{u}n_k$. Therefore, we have the following formulae: $$\frac{\partial }{\partial y^k} (\phi(\phi-s\phi_s))=\frac{1}{u}(\phi \phi_s-s\phi_s^2-s\phi\phi_{ss})n_k=\frac{\mu}{u}(x_k-\frac{s}{u}y_k),$$
$$\frac{\partial }{\partial y^k} (\phi_s^2+\phi \phi_{ss})=\frac{1}{u}(3\phi_s\phi_{ss}+\phi\phi_{sss})n_k=\frac{\nu}{u}(x_k-\frac{s}{u}y_k),$$
$$\frac{\partial }{\partial y^k} (\phi \phi_s-s\phi_s^2-s\phi\phi_{ss})=-\frac{s}{u}(3\phi_s\phi_{ss}+\phi\phi_{sss})n_k=-\frac{s\nu}{u}(x_k-\frac{s}{u}y_k).$$
Using the above formulas and \eqref{Eq:derivatives}, then differentiating the metric tensor given in \eqref{Eq:g_ij} with respect to $y^k$ we obtain the  Cartan tensor $C_{ijk}$.
\end{proof}
Now, we are in a position to calculate the main scalar $I$.
\begin{theorem}\label{Th:main_scalar}
The main scalar $I$ of a spherically symmetric Finsler surface is given by
\begin{equation}
\label{Eq:Main_Scalar}
\begin{split}
I=&\frac{\phi}{2}\left( \frac{3\mu}{a}((m^1)^2+(m^2)^2)  -3a\mu  (r^2-s^2)^2 B^2 +\frac{\nu}{a^3}\right),
\end{split}
\end{equation}
where  $B:=\rho_2+s\rho_3$.
\end{theorem}
\begin{proof}
We have the following two formulas
$$m^ix_i=a(r^2-s^2)(\rho_0+s\rho_2+r^2\rho_3)=a(r^2-s^2)A,$$
$$m^iy_i=a u (r^2-s^2)(\rho_2+s\rho_3)=a u (r^2-s^2)B,$$
where $A:=\rho_0+s\rho_2+r^2\rho_3$.
Substituting the above  formulas and the Cartan tensor given in Proposition  \ref{Prop:Cartan_tensor} into the formula
$$I=Fm^im^jm^kC_{ijk}=u \phi m^im^jm^kC_{ijk}$$
we get the formula
\begin{equation}
\label{Eq:I_1}
\begin{split}
I=&\frac{\phi}{2}\big{(}a^3(r^2-s^2)^3\left(\nu A^3-3(\mu-s^2\nu)AB^2+(3s\mu -s^3\nu)B^3-3s\nu A^2B\right)\\
&+3 a \mu (r^2-s^2)(A-sB) m^im^j\delta_{ij}\big{)}.
\end{split}
\end{equation}
Straightforward calculations imply the property $$a^2(r^2-s^2)(A-sB)=1.$$
Using the above property together with the fact that
$$A^3-s^3B^3=(A-sB)(A^2+sAB+s^2B^2)$$
 one can simplify  \eqref{Eq:I_1} together with the fact that $ m^im^j\delta_{ij}=(m^1)^2+(m^2)^2$ we get  the  required  formula of the main scalar as in \eqref{Eq:Main_Scalar}.
\end{proof}

\begin{theorem}
For a  spherically symmetric Finsler  surface $F=u\phi(r,s)$,   $I=0$ if and only if
  $\mu=0$. That is, $F$ is Riemannian if and only if  $\mu=0$.
\end{theorem}
\begin{proof}
It is known that   a Riemannian surface is characterized by the vanishing of its main scalar. Now, if $F=u\phi(r,s)$ is Riemannian, then $F^2=u^2\phi^2$ is quadratic in $y$. That is, $\phi$ must be in the form
$$\phi=\sqrt{f_1(r)s^2+f_2(r)}$$
where $f_1(r)$ and $f_2(r)$ are arbitrary functions.
This implies that $\mu=\phi \phi_s-s\phi_s^2-s\phi\phi_{ss}=0$.

Conversely, if $\mu=0$, then using the fact  that $\displaystyle{\frac{\partial \mu}{\partial s}}=-s \nu$, we obtain that $\nu=0$. By substituting $\mu=0$ and $\nu=0$ into $I$, we get $I=0$.  This completes the proof.
\end{proof}

We end this section by the following examples:
\begin{example}
Let $F$ be a Randers surface of a spherically symmetric Finsler surface, that is, 
  $$F=u(1+s), \quad \phi(s)=1+s.$$
  One can use the Maple program to do most of the following calculations. 
  The supporting elements $\ell_1$ and $\ell_2$ are given by:
  $$\ell_1=\frac{\partial F}{\partial y^1}=\frac{\phi}{u} y^1+\phi_s\left(x^1-\frac{s}{u}y^1\right)=x^1+\frac{y^1}{u}.$$
  $$\ell_2=\frac{\partial F}{\partial y^2}=\frac{\phi}{u} y^2+\phi_s\left(x^2-\frac{s}{u}y^2\right)=x^2+\frac{y^2}{u}.$$
  Using \eqref{Eq:a}, we have
  $a(r,s)=\sqrt{\frac{1+s}{r^2-s^2}}.$
  Then, we have
  $$m_1=an_1= \left( x^1-\frac{s}{u} y^1 \right)\sqrt{\frac{1+s}{r^2-s^2}}, \quad m_2=a n_2= \left( x^2-\frac{s}{u} y^2\right)\sqrt{\frac{1+s}{r^2-s^2}}.$$
  The functions $\rho_0$, $\rho_1$, $\rho_2$ and $\rho_3$ are given by:
  $$\rho_0=\frac{1}{1+s},\quad \rho_1=\frac{r^2+s}{(1+s)^3},\quad \rho_2=-\frac{1}{(1+s)^2},\quad \rho_3=0.$$
  Then, by \eqref{Eq:Berwald_frame}, we have  $m^i$ are given as follows:
  $$\ell^1=\frac{y^1}{F}=\frac{y^1}{u(1+s)}, \quad  \ell^2=\frac{y^2}{F}=\frac{y^2}{u(1+s)},$$
  $$m^1=-\frac{r^2y^1-s u x^1+sy^1-ux^1}{u(1+s)^2}\sqrt{\frac{1+s}{r^2-s^2}}  ,  \quad m^2= -\frac{r^2y^2-s u x^2+sy^2-ux^2}{u(1+s)^2}\sqrt{\frac{1+s}{r^2-s^2}}.$$
  To obtain the main scalar, we have
  $$\mu=\phi \phi_s-s\phi_s^2-s\phi\phi_{ss}=1, \quad \nu=3\phi_s\phi_{ss}+\phi\phi_{sss}=0,\quad B=\rho_2+s\rho_3=-\frac{1}{(1+s)^2}.$$
  Now, by substituting the above formulae into \eqref{Eq:Main_Scalar}, we get the  main scalar, that is, 
  $$I=-\frac{3}{2}\frac{r^2-2s^2-2s-2}{(1+s)^2}\sqrt{\frac{r^2-s^2}{1+s}}.$$
\end{example}

\begin{example}
 Let $F=u\phi(r,s)$, where
\begin{align*}
  \phi (r,s)
    &=\sqrt{1+s^2}.
\end{align*}
Using, Maple program,  \eqref{P,Q} and the above formula of  $\phi(r,s)$,  we obtain that
\begin{equation}
\label{P_Q_from_phi}
P= 0, \quad Q=\frac{1}{2(1+ r^2)}.
\end{equation}
Also, we have
$$R_1=\frac{1}{1+ r^2},\quad R_2=-\frac{1}{1+ r^2},\quad R_3=-\frac{1}{(1+ r^2)^2},\quad R_4=\frac{s}{(1+ r^2)^2},\quad R_5=0.$$
For this surface,  we have
$K=\frac{1}{(1+r^2)^2}$.
\end{example}

\section{A note on a general $(\alpha,\beta)$-metric}

 Let  $\alpha=\sqrt{a_{ij}y^iy^j}$ be a Riemannian metric and $\beta=b_i(x)y^i$ be a $1$-form on $M$.
 A regular general $(\alpha,\beta)$-metric is defined by a  Finsler function   $F$  on the form
 $$F=\alpha \varphi(b^2,s),\  s=\frac{\beta}{\alpha}$$
 such that $\varphi$  is a positive smooth function, $ \|\beta_x\|_\alpha<b_0$ and

 	\begin{equation}\label{phi_conditions}
 	 \varphi(b^2,s)-s\varphi_s(b^2,s)>0, \quad \varphi(b^2,s)-s\varphi_s(b^2,s)+(t^2-s^2)\varphi_{ss}(b^2,s)>0, \quad |s|<t<b_0.
 	\end{equation}
 For more details cf. \cite{Yu-Zhu}.

An $(\alpha,\beta)$-metric is a special class of  general $(\alpha,\beta)$-metric, namely, when $\varphi$ is independent of $b^2$.
 Another interesting particular case, $F$ is  {spherically symmetric metric} when $\alpha=|y|$ is standard  Euclidean norm and $\beta=\langle x , y \rangle$ is the standard  Euclidean inner product.

For a general $(\alpha,\beta)$-metric $F=\alpha\varphi(s)$,   the determinant of the metric tensor  $g_{ij}$ is given by
\begin{equation}
\label{Eq:det(g)}
\det(g_{ij})=\varphi^{n+1}(\varphi-s\varphi_s)^{n-2}(\varphi-s\varphi_s+(b^2-s^2)\varphi_{ss})\det(a_{ij}),
\end{equation}
It is clear that the condition  \eqref{phi_conditions} is among the regularity conditions of the  Finsler function $F=\alpha \varphi(b^2,s)$. But in a wider context or applications we may require $F$ non-regular and so we may lose the condition \eqref{phi_conditions}. The following theorem shows how we can choose $\varphi(b^2,s)$.
\begin{theorem}
 \label{Excluding}
We have the following equivalences:
\begin{description}
\item[(a)] $ \varphi-s\varphi_s=0\Longleftrightarrow\varphi=f(b^2)s$,
\item[(b)]$\varphi\varphi_s-s(\varphi_s^2+\varphi \varphi_{ss})=0 \Longleftrightarrow \varphi(b^2,s)=f_1(b^2)s+f_2(b^2)\sqrt{b^2-s^2}$,
\end{description}
{ where} $ f_1$   and $ f_2$  are  arbitrary functions of $b^2$. Moreover, in both cases such choices for $\varphi$ lead to a degenerate metric; that is, $\det(g_{ij})=0$. Therefore,    these choices should be excluded.
\end{theorem}

\begin{proof} 
\noindent (a)  Let $ \varphi-s\varphi_s=0$, then we have $\frac{\varphi_s}{\varphi}=\frac{1}{s}.$
The solution of this  differential equation is  $\varphi(b^2,s)=f(b^2)s$.
Therefore,  the Finsler metric $F$ is given by
$$F=\alpha \varphi=\alpha f(b^2) \frac{\beta}{\alpha}=f(b^2) \beta.$$
Since the $1$-form $\beta$ is linear in $y$, then   the above formula of $F$ leads to a degenerate metric, that is, $\det(g_{ij})=0$ which   impossible and hence $ \varphi-s\varphi_s\neq 0$. If $\varphi(b^2,s)=f(b^2)s$, then we have, obviously, $\varphi-s\varphi_s=0$ and the proof of (a)  is done.

\noindent (b) Now, assume the differential equation
$$\varphi-s\varphi_s+(b^2-s^2)\varphi_{ss}=0.$$
By (a) we have  $ \varphi-s\varphi_s\neq 0$, then we can write the above equation as follows
$$\frac{-s\varphi_{ss}}{\varphi-s\varphi_s}=\frac{s}{b^2-s^2}.$$
Since $(\varphi-s\varphi_s)_s=-s\varphi_{ss}$, then we have
$$\varphi-s\varphi_s=\frac{f(b^2)}{\sqrt{b^2-s^2}}.$$
Which has the solution
$$\varphi(b^2,s)=f_1(b^2)s+f_2(b^2)\sqrt{b^2-s^2}$$
where $f_1 $ and $ f_2 $ are  arbitrary functions of $b^2$. Moreover, since the determinant of $g_{ij}$ is given by
$$\det(g_{ij})=\varphi^{n+1}(\varphi-s\varphi_s)^{n-2}(\varphi-s\varphi_s+(b^2-s^2)\varphi_{ss})\det(a_{ij}), $$
then we conclude that $\det(g_{ij})=0$. Conversely, if   $\varphi(b^2,s)=f_1(b^2)s+f_2(b^2)\sqrt{b^2-s^2}$, then we get that $\varphi\varphi_s-s(\varphi_s^2+\varphi \varphi_{ss})=0$. Which completes the proof.
\end{proof}

\subsection{Special cases and examples}
Since many examples  appeared in the literature look like the metric given in Theorem \ref{Excluding} (b), we have to be careful when we deal with such a metric. Here we mention some papers where the this metric appeared.

  \medskip

In \cite[Example 9.2]{Example-9.2}, the   general $(\alpha,\beta)$-metric with the function $\varphi$ given by
$$
\varphi\left(b^{2}, s\right)=\frac{\sqrt{b^{2}-s^{2}}}{b^{2}}.
$$
It was mentioned, in \cite[Example 9.2]{Example-9.2}, that the metric $F=\alpha \varphi\left(b^{2}, s\right)$ is locally projectively flat with vanishing flag curvature. According to Theorem \ref{Excluding}  (b), the above function $\varphi$ is just a special case  by the choice $f_1(b^2)=0$ and $f_2(b^2)=\frac{1}{b^2}$. That is, the metric $F$ is degenerate on the whole tangent bundle.

 \medskip

In \cite[Theorem 1.1]{general_3}, it was mentioned that the metric
$$
F=\sqrt{f(b^{2}) \alpha^{2}+g(b^{2}) \beta^{2}}+h(b^{2}) \beta
$$
 is of isotropic Berwald curvature.
 According to Theorem \ref{Excluding},     the choice    $f\left(b^{2}\right) =b^2$ and $g\left(b^{2}\right) =-1$ should be excluded.
One can find more papers where the case under consideration appears, for example, we refer to \cite{general_1,general_2,general_3}.

 \medskip

In \cite[Theorem 1.1]{Theorem-1.1-Shen}, it was stated that   the metric
$$F=\sqrt{\alpha^{2}+k \beta^{2}}+\epsilon \beta$$
 is projectively flat with constant flag curvature $K<0$, where $k$ and $\epsilon \neq 0$ are constants.
 According to Theorem \ref{Excluding}, if $b^2=1$ and $k=-1$, then the above metric is degenerate.
 See \cite[Theorem 1.1]{Cheng-Shen} $F=k_1\sqrt{\alpha^{2}+k_2 \beta^{2}}+k_3 \beta$ is   of isotropic
S-curvature under certain conditions, see also \cite{Shen}.

 \medskip

In \cite[Theorem 1.2]{Taybi}, it was stated that all Randers metrics on the form
$$
F=u\left(\sqrt{p(r) {s}^{2}+q(r)}+c(r) {s}\right) \quad \text { or } \quad F=\frac{u \ t(r)}{c}
$$
are $L$-reducible metrics. According to Theorem \ref{Excluding}, we have to exclude the choices that lead to $p(r) s^{2}+q(r)=f(r)(r^2-{s}^2)$.
More articles in the context of spherically symmetric metrics such that the above situation appears. For example, we refer to  \cite{E.Guo-etal,L.Zhou}.

{
\section*{Acknowledgment}
The author  would like to thank the reviewers for their constructive
comments and recommendations that improved the paper.}

\section*{Declarations}
\begin{itemize}
\item \textbf{Competing interests}: The author  declares no conflict of interest.
  \item \textbf{Availability of data and material}: Not applicable.
  \item \textbf{Funding}: Not applicable.
  \item \textbf{Authors' contributions}: The author  has made substantive contributions
to the article and assume full responsibility for its content. The author  read and approved the final manuscript.

\end{itemize}

\end{document}